\documentclass[11pt]{amsart}
\providecommand{\XLeftMargin}{2.5cm}
\providecommand{\XTopMargin}{2.5cm}
\providecommand{\XRightMargin}{2.5cm}
\providecommand{\XBottomMargin}{2.0cm}

\usepackage{amssymb,latexsym,amsmath,amsthm,amscd}
\usepackage[all]{xy}
\xyoption{arc}
\providecommand{\XLeftMargin}{2cm}
\providecommand{\XTopMargin}{1.8cm}
\providecommand{\XRightMargin}{2cm}
\providecommand{\XBottomMargin}{1.8cm}

\usepackage[left=\XLeftMargin,top=\XTopMargin,right=\XRightMargin,bottom=\XBottomMargin]{geometry}
\usepackage{graphicx}
\usepackage{todonotes}

\newlength\XXXMyLength\makeatletter
\@ifclassloaded{amsart}{\def\Xadjustleft[#1]{}
}{
\def\Xadjustleft[#1]{\setlength\XXXMyLength{#1}\ifnum\numexpr\leftmargin>\numexpr\XXXMyLength\hspace{-\XXXMyLength}\else\hspace{-\leftmargin}\fi}
}
\makeatother

\UseComputerModernTips

\theoremstyle{plain}
\newtheorem{thm}{Theorem}[section]
\newtheorem*{thm*}{Theorem}
\newtheorem{lemma}[thm]{Lemma}

\newtheorem*{lem*}{Lemma}
\newtheorem{prop}[thm]{Proposition}
\newtheorem*{prop*}{Proposition}

\newtheorem{cor}[thm]{Corollary}
\newtheorem*{cor*}{Corollary}

\newtheorem*{conj*}{Conjecture}

\theoremstyle{definition}

\newtheorem*{cons*}{Construction}
\newtheorem{df}[thm]{Definition}
\newtheorem*{df*}{Definition}
\newtheorem{nota}[thm]{Notation}
\newtheorem*{nota*}{Notation}

\newtheorem*{qu*}{Question}

\newtheorem{rmk}[thm]{Remark}
\newtheorem*{rmk*}{Remark}

\newtheorem{ex}[thm]{Example}
\newtheorem*{ex*}{Example}

\newcommand{\bN}{\mathbb{N}}

\newcommand{\cE}{\mathcal{E}}
\newcommand{\cF}{\mathcal{F}}

\newcommand{\cO}{\mathcal{O}}

\def\del{\partial}

\def\sumprime{\mathop{\sum{\raise3pt\hbox{${}'$}}}}

\makeatletter
\def\revddots{\mathinner{\mkern1mu\raise\p@
\vbox{\kern7\p@\hbox{.}}\mkern2mu
\raise4\p@\hbox{.}\mkern2mu\raise7\p@\hbox{.}\mkern1mu}}
\makeatother

\providecommand{\abs}[1]{\left\vert #1 \right\vert}

\newcommand{\comment}[1]{}

\def\constC{{\delta}}
\def\constD{{\lambda}}
\def\chern{{\mathrm{c}}}
\def\chernchar{{\mathrm{ch}}}
\def\chernrnk{{\rm rk}}

\usepackage{enumitem}

\begin{document}

\title[Riemann-Hurwitz for the Algebraic Euler Characteristic]{A Riemann-Hurwitz Theorem for the Algebraic Euler Characteristic}
\author{Andrew Fiori}

\email{andrew.fiori@ucalgary.ca}
\address{Mathematics \& Statistics
612 Campus Place N.W.
University of Calgary
2500 University Drive NW
Calgary, AB, Canada
T2N 1N4}

\begin{abstract}
We prove an analogue of the Riemann-Hurwitz theorem for computing Euler characteristics of pullbacks of coherent sheaves through finite maps of smooth projective varieties in arbitrary dimensions, subject only to the condition that the irreducible components of the branch and ramification locus have simple normal crossings. 
\end{abstract}

\thanks{This work was done while the author was a Fields postdoctoral researcher at Queen's University}

\maketitle

\section{Introduction}

Consider a finite map $\pi : X \rightarrow Y$ of degree $\mu$.
Let $B = \cup B_i$ be the branch locus and its irreducible decomposition.
Let $R= \pi^{-1}( B )= \cup R_j$ be the ramification locus and the irreducible decomposition of its reduction. Note that we are taking here the potentially non-standard choice to include in $R$ even those components of $\pi^{-1}(B)$ which are not ramified, this convention will be consistent throughout.
The Riemann-Hurwitz formula for the topological Euler characteristic of curves can roughly be interpreted to say:
\[ \chi(X) - \mu \cdot \chi(Y) = \sum_i r_i\chi(R_i). \]
For some integers $r_i$ determined by local data.
This formula can be generalized both to higher dimensional manifolds, but also to the algebraic Euler characteristic.
However, in the higher dimensional algebraic setting, such a formula typically requires additional hypothesis on the ramification and/or branch locus such as:
\begin{itemize}
\item The ramification locus to be non-singular.
\item The irreducible components of the branch locus do not intersect.
\item The irreducible components of the ramification locus to be non-singular.
\item The formula is cleanest if the irreducible components of the ramification locus have trivial self intersection.
 
        Note that the work of Izawa \cite{IzawaRH} handles the case where this last condition is not true but requires the previous conditions.
\end{itemize}

We would like to be able to reduce these conditions to the requirement that the branch and ramification locus consist of divisors with simple normal crossings.

The result which we obtain is a formula of the form:
\[ \chi(X) - \mu\cdot\chi(Y) = \sum_\alpha r_\alpha\chi(R^\alpha). \]
Where the $R^\alpha$ are irreducible components of the (possibly repeated) intersections, that is the strata, of the ramification locus.
The $r_\alpha$ are constants defined in terms of the ramification structure along $R^\alpha$ by universal equations determined by $\alpha$.
Precise statements are given in Theorem \ref{thm:main} and Corollary \ref{cor:main} noting that terminology introduced elsewhere will be necessary to understand them. Propositions \ref{prop:simpleMF} and \ref{prop:simpleNMF} give alternative expressions for some of the coefficients $r_\alpha$.
Other theorems which may be of interest are Theorems \ref{thm:pullbacklog} and \ref{thm:pullbacklogchern} which describe the functoriality of the pullback of logarithmic Chern classes and the logarithmic Euler characteristic through finite maps. It is likely both of these results admit generalizations outside the context in which the author is able to prove them.

Moreover, our argument works virtually identically in each of the following cases:
\begin{itemize}
\item $\chi(X)$ is the topological Euler characteristic, in this case the above results are then classical and follow from excision.

\item $\chi(X) = \chi(X,\cO_X)$ is the algebraic Euler characteristic.

          The result is well known when the map is \'etale (see \cite[Ex. 18.3.9]{FultonIntersection1}).
          The case of no intersection between components of branch/ramification locus is handled for example in the work of Izawa \cite{IzawaRH}.
\item $\chi(X) = \chi(X,\cF)$ is the Euler characteristic of a coherent sheaf $\cF$.

\item The same argument should apply formally to any `characteristic' defined by a multiplicative sequence on the Chern classes.
The formulas for the coefficients $r_\alpha$ do naturally depend on this choice of characteristic.
\end{itemize}
We shall only present the argument for the case of $\chi(X,\cF)$, the main results are Theorem \ref{thm:main} as well as Corollary \ref{cor:main}.
The strategy of proof uses primarily formal properties of logarithmic Chern classes and formal properties of multiplicative sequences.
The paper is organized as follows:

\begin{itemize}
\item In Section \ref{sec:bgnot} we introduce our notation and the key results we shall make use of. This includes in particular Lemmas \ref{lem:cute1}-\ref{lem:cute3} and Theorem \ref{thm:pullbacklog}.
\item In Section \ref{sec:logeuler} we introduce our definition of logarithmic Euler characteristic.
\item Section \ref{sec:logveuler} contains the key calculations, which compares the classical Euler characteristic to the logarithmic Euler characteristic.
\item Section \ref{sec:rh} applies the results of Section \ref{sec:logveuler} to the problem of giving the Riemann-Hurwitz theorem discussed above.
\item In Section \ref{sec:logeulerself} we discuss computing the contribution to the logarithmic Euler characteristics of the `self-intersection' terms. 
\end{itemize}

We should mention that our original motivation for considering the objects being introduced is to compute dimension formulas for spaces of modular forms. For this application it is actually the results of Section \ref{sec:logveuler} and Section \ref{sec:logeulerself} that by way of the work of Mumford in \cite{Mumford_Proportionality} play a significant role. Though actual dimension formulas require additional arithmetic and/or combinatorial input the results of the aforementioned sections can be seen as a generalization of a key ingredient for the approach used in \cite{Tsushima}.

\section{Background and Notation}\label{sec:bgnot}

\begin{nota}\label{not:1}
We shall make use of the following notation.
\begin{enumerate}
\item $X$ and $Y$ shall always be varieties, typically assumed to be smooth and projective.

\item
Given a variety $X$ we shall denote by 
\[ \Omega^1_X \]
the cotangent bundle of $X$.
\item $\Delta = \cup_i \{ D_i \}$ shall always be a collection of (reduced irreducible) divisors on a variety.
These shall typically be assumed to have simple normal crossings.

\item
Given $\Delta = \cup_i \{D_i\}$ a collection of divisors on $X$ we shall denote by:
\[ \Omega^1_X(\log \Delta) \]
the logarithmic cotangent bundle of $X$ relative to $\Delta$,
\item
For any $Y\subset X$ we shall denote by:
\[ \Omega^1_Y(\log \Delta') \]
the logarithmic cotangent bundle of $Y$, relative to $\Delta' = \cup_i \{D_i\cap Y\}$, where we consider only those $i$ such that $Y\not\subset D_i$.
When we write this we shall always assume that $Y$ meets the relevant $D_i$ transversely.
Whenever we write $\Delta'$, the relevant $Y$ shall be understood.
\item
Given any coherent sheaf $\cF$ on $X$ we shall denote by:
\[ \chern(\cF)  = \sum_i \chern_i(\cF) \]
the total Chern class and the $i$th Chern class (see \cite[Ch. 3]{FultonIntersection1}).

We shall denote by $\chernchar(\cF)$ and ${\rm Todd}(\cF)$ the Chern character and Todd class respectively.
The Todd class ${\rm Todd}(\cF)$ has a universal expression in terms of the $\chern_i(\cF)$, whereas $\chernchar(\cF)$ additionally requires the rank, $\chernrnk(\cF)$, specifically the constant part of the Chern character.
These classes can be interpreted as being in the cohomology ring or the Chow ring as appropriate from context.

We can interpret $\chernchar(\cF)$ as a vector determining all of $\chernrnk(\cF), \chern_1(\cF),\ldots, \chern_n(\cF)$.
Conversely, given a vector $\underline{x} = (x_0,\ldots,x_n)$ we shall write $ \chernchar(\underline{x})$ to indicate the formal expression in the $x_i$ where we replace $\chern_i$ by $x_i$ and $\chernrnk(\cF)$ by $x_0$ in the formal expression for $\chernchar(\cF)$. For brevity, and to make clear the connection to the role of the Chern character, we shall often write $\chernchar(\underline{x})$ or $\chernchar(\cF)$ when evaluating a function on the vectors $(x_0,\ldots,x_n)$ or $(\chernrnk(\cF),\chern_1(\cF),\ldots,\chern_n(\cF))$ when it is defined through $\chernchar(\cF)$ (see for example Theorem \ref{thm:RR}).

\item
Given $\Delta = \cup_i \{D_i\}$ a collection of divisors on $X$ we shall denote by:
\[ \Delta_k \]
the $k$th elementary symmetric polynomial in the $D_i$.
so that:
 \[ \prod_i (1-D_i) = \sum_k (-1)^k \Delta_k. \]
The products above take place in either the cohomology ring or the Chow ring as appropriate from context.
\item
When we say that $\alpha$ is a partition of $m$ we mean that $m = \sum_i \alpha_i i $. Given a partition $\alpha$, we shall denote by $\abs{\alpha}$ the value $m$ it is partitioning. That is $\abs{\alpha} = \sum_i \alpha_i i$. Moreover, given such a partition we shall denote by:
\[ \chern^\alpha(\cF) = \prod_i \chern_i(\cF)^{\alpha_i} \]
and by:
\[ \Delta^\alpha =  \prod_i \Delta_i^{\alpha_i}. \]
\item
Given a monomial exponent $\underline{b} = (b_1,\ldots,b_\ell) \in \bN^\ell$ of total degree $\abs{\underline{b}} = \sum b_i$ we shall denote by:
\[ D^{\underline{b}} = \prod D_i^{b_i}. \]
The products above take place in either the cohomology ring or the Chow ring as appropriate from context.
Whenever we write this, the choice of base $D$ will make clear the relevant $\Delta$ to which $D_i$ belong.
Do not confuse $D^\ell$ with $D^{\underline{b}}$, the former will always be the self intersection of a particular divisor $D \in \Delta$.
\end{enumerate}
\end{nota}

\begin{thm}[Riemann-Roch Theorem]\label{thm:RR}
For each $n\in \bN$ there is a universal polynomial 
\[ Q_n(x_0,\ldots,x_n; y_1,\ldots,y_n) = Q_n(\chernchar(\underline{x});y_1,\ldots,y_n) \]
 such that for all smooth projective varieties $X$ of dimension $n$ and coherent sheaves $\cF$ on $X$ the Euler characteristic of $\cF$ is:
\[ \chi(X,\cF) = Q_n(\chernrnk(\cF), \chern_1(\cF),\ldots, \chern_n(\cF);  \chern_1(\Omega^1_{X}),\ldots, \chern_n(\Omega^1_{X})) = Q_n(\chernchar(\cF);  \chern_1(\Omega^1_{X}),\ldots, \chern_n(\Omega^1_{X})). \]
The polynomial is given explicitly by:
\[  Q_n( \chernrnk(\cF),\chern_1(\cF),\ldots, \chern_n(\cF);  \chern_1(\Omega^1_{X}),\ldots, \chern_n(\Omega^1_{X})) =  \deg_n(\chernchar(\cF){\rm Todd}(\Omega^1_X)). \]
Recall that we interpret $\chernchar(\cF)$ as a vector determining all of $\chernrnk(\cF), \chern_1(\cF),\ldots, \chern_n(\cF)$ and $\chernchar(\underline{x})$ as the corresponding vector where the $x_i$ are substituted in the universal expression for $\chernchar(\cF)$.
\end{thm}

We have the following explicit formulas for $Q_n$ for small $n$:
\begin{align*}
  Q_0(x_0;y_0) &= x_0 \\
  Q_1(x_0,x_1; y_1) &= \frac{1}{2}x_0y_1 + x_1\\
  Q_2(x_0,x_1,x_2; y_1,y_2) &= \frac{1}{12}x_0(y_1^2+y_2) + \frac{1}{2}x_1y_1 +  \frac{1}{2}(x_1^2-2x_2) \\
  Q_3(x_0,x_1,x_2,x_3; y_1,y_2,y_3) &= \frac{1}{24}x_0y_1y_2 + \frac{1}{12}x_1(y_1^2+y_2)+ \frac{1}{4}(x_1^2-2x_2)y_1 + \frac{1}{6}(x_1^3-3x_1x_2+3x_3) \\
 Q_4(x_0,\ldots,x_4; y_1,\ldots,y_4)   &=  \frac{1}{720}x_0(-y_1^4 + 4y_1^2y_2 + y_1y_3 + 3y_2^2 - y_4) +  \frac{1}{24}x_1y_1y_2 + \cdots 
\end{align*}

\begin{rmk}
The most important feature of the explicit description we shall make use of is that 
\[ {\rm Todd}(\cE_1 \oplus \cE_2) = {\rm Todd}(\cE_1) {\rm Todd}(\cE_2) \]
 so that $Q_n$ is effectively multiplicative in the $\chern_i(\Omega^1_{X})$ set of parameters.
\end{rmk}

The following proposition makes precise what we mean by multiplicative.
\begin{prop}\label{prop:multQ}
For notational convenience in the following we use the constants $u_0=v_0=1$ and $u_i=v_i=0$ for $i<0$.
Consider formal variables $u_1,\ldots, u_n$ and $v_1,\ldots,v_n$ and set $y_i = \sum_{j+k = i} u_jv_k$ then
\[ Q_n(\chernchar(\underline{x}); y_1,\ldots,y_n) = \sum_{\ell + m = n} Q_\ell(\chernchar(\underline{x}); u_1,\ldots,u_\ell) Q_m(1; v_1,\ldots,v_m). \]
\end{prop}
\begin{proof}
Denote by ${\rm Todd}(\underline{y}),\,{\rm Todd}(\underline{u}),\,{\rm Todd}(\underline{v})$ the universal expression for the Chern character or Todd class where we substitute the appropriate set of variables for the Chern classes.
We then have:
\begin{align*}
  Q_n(\chernchar(\underline{x}); y_1,\ldots,y_n) &= \deg_n( \chernchar(\underline{x}) {\rm Todd}(\underline{y}))\\
                                                             &=  \deg_n( \chernchar(\underline{x}) {\rm Todd}(\underline{u}){\rm Todd}(\underline{v}))\\
                                                              &= \sum_{\ell+m=n} \deg_\ell( \chernchar(\underline{x}) {\rm Todd}(\underline{u}))\deg_m({\rm Todd}(\underline{v}))\\
                                                              &= \sum_{\ell + m = n} Q_\ell(\chernchar(\underline{x}); u_1,\ldots,u_\ell) Q_m(1; v_1,\ldots,v_m).\qedhere
\end{align*}
\end{proof}
\begin{rmk}
The same formula holds if we use instead the system of polynomials
\[ Q_n(x_0,x_1,\ldots,x_n; y_1,\ldots,y_n) = y_n \]
 which give the topological Euler characteristic. The algebraic Euler characteristic of $X$ is just the special case of $\cF = \cO_X$.
\end{rmk}

\begin{nota}\label{not:2}
We shall also need the following terminology and combinatorial quantities. Note that these are all universal and depend only on the choice of multiplicative sequence $Q$. These constants can all be effectively computed.
\begin{enumerate}
\item

Given any monomial exponent $\underline{b}$ we shall denote by:
\[ \constC_{\underline{b}} \]
the coefficient of $D^{\underline{b}}$ in $Q_{\abs{b}}(1;\Delta_1,\ldots,\Delta_{\abs{\underline{b}}})$. Note that these coefficients depend only on the monomial type of $\underline{b}$, that is the multi-set $\{ b_i \neq 0 \}$. In particular $\constC_{(2,0,1)}  = \constC_{(1,2,0)} = \constC_{(2,1)}$.

In the context we are working, where $Q$ describes the algebraic Euler characteristic, this is also precisely the coefficient of $D^{\underline{b}}$ in:
\[ \prod_{D\in \Delta} \frac{D}{1-e^{-D}}. \]

For example, given that:
\[ Q_2(1,\Delta_1,\Delta_2) = \frac{1}{12}(\Delta_1^2+\Delta_2) = \frac{1}{12}\sum_i D_i^2 + \frac{1}{4}\sum_{i\neq j} D_iD_j \]
we have that:
\[  \constC_{(2)} = \frac{1}{12} \qquad  \constC_{(1,1)} = \frac{1}{4}. \]
Likewise given that:
\[ Q_3(1,\Delta_1,\Delta_2,\Delta_3) = \frac{1}{24}\Delta_{1}\Delta_2 = 0 \sum_i  D_i^{3} + \frac{1}{24}\sum_{i\neq j} D_i^2D_j  + \frac{1}{8}\sum_{i\neq j\neq k} D_iD_jD_k \]
we have that:
\[ \constC_{(3)} = 0 \qquad \constC_{(2,1)} = \frac{1}{24} \qquad \constC_{(1,1,1)} =  \frac{1}{8}. \]
We can likewise compute that:
\[ \constC_{(0)} = 1 \qquad \constC_{(1)} = \frac{1}{2}.  \]
\item
We may think of the monomial exponents $\underline{b}$ as vectors indexed by the elements $D$ of $\Delta$.
As such, given two monomial exponents $\underline{b}$ and $\underline{b}'$ we shall denote $\underline{b} \leq \underline{b}'$ if the inequality holds component wise so that we may write:
\[ D^{\underline{b}}D^{\underline{b}''} = D^{\underline{b}'} \]
for some $\underline{b}''$ with all components $b_i''\ge0$.
By the support of a monomial exponent $\underline{b}$ we mean the collection of $D_i$ for which $b_i\neq 0$. We say $\underline{a}$ and $\underline{b}$ have disjoint support if the corresponding collections have no common elements.
Given a monomial exponent $\underline{b}$ we shall say it is multiplicity free, abbreviated MF, if $b_i \leq 1$ for all $i$, otherwise, we shall say it is not multiplicity free, abbreviated NMF.
Note that a monomial exponent is MF precisely when computing $D^{\underline{b}}$ involves no self intersections.
Finally, given a collection of monomial exponents $\underline{b}_j$ we shall write:
\[ \sum_j \underline{b}_j = \underline{b} \]
if this is true as a vector sum.

\begin{prop}\label{prop:constC}
If $\underline{b}_j$ have disjoint support then:
\[ \constC_{\sum_j \underline{b}_j} = \prod_j \constC_{\underline{b}_j}. \]
\end{prop}
\begin{proof}
This follows immediately from the multiplicativity of $Q$ as in Proposition \ref{prop:multQ} and the observation that:
\[ \chern_i \left( \underset{D \in \Delta}\oplus  \cO(D)\right) = \Delta_i. \qedhere\]
\end{proof}

\item Given a monomial exponent $\underline{b}$ denote by $\tilde{\underline{b}}$ the monomial exponent such that \[ \tilde{b}_i = \min(1,b_i),\] so that $\tilde{\underline{b}}$ captures the support of $\underline{b}$ but $\tilde{\underline{b}}$ is MF.
For example $\widetilde{(1,2,3)} = (1,1,1)$.
Moreover, we shall denote by $\hat{\underline{b}}$ the monomial exponent such that 
\[ \hat{b}_i = \begin{cases} 1 & b_i = 1 \\ 0 & \text{otherwise}, \end{cases} \]
 so that $\hat{\underline{b}}$ captures the part of the support of $\underline{b}$ where $\underline{b}$ has no self intersection.
For example $\widehat{(2,1,3)} = (0,1,0)$.

\item Given a monomial exponent $\underline{b}$ we shall denote by:

 \[ \constD_{\underline{b}} =  \sum_{k\ge 0} (-1)^{k+1}  \underset{\sum {\underline{b}_j} = \underline{b}}{\sum_{(\underline{b}_1,\ldots,\underline{b}_k)}} \left(\prod_{j=1}^k \constC_{\underline{b}_j}\right) = \constC_{\underline{b}} \sum_{k\ge 0}(-1)^{k+1}  \underset{\sum {\underline{b}_j} = \underline{b}}{\sum_{(\underline{b}_1,\ldots,\underline{b}_k)}} 1. \]
In the summation we consider only terms with all $\abs{\underline{b}_j} \ge 1$ and where in the tuple $(\underline{b}_1,\ldots,\underline{b}_k)$ all of $\underline{b}_j$ have disjoint support and each of $\underline{b}_1,\ldots,\underline{b}_{k-1}$ are MF, so that only $\underline{b}_k$ is potentially NMF.
Note, when $\underline{b}$ is MF, these last three conditions are automatic. For $k$ sufficiently large the inner sum is an empty sum.
Under these conditions the equality between the two definitions is immediate from Proposition \ref{prop:constC}.
\begin{prop}\label{prop:constD}
When $\underline{b}$ is MF we have:
\[   \sum_{k\ge 0} (-1)^{k+1}  \underset{\sum {\underline{b}_j} = \underline{b}}{\sum_{(\underline{b}_1,\ldots,\underline{b}_k)}} 1 = (-1)^{\abs{\underline{b}}} \]
where the sum is taken as above.
\end{prop}
\begin{proof}
Each tuple $(\underline{b}_1,\ldots,\underline{b}_k)$ contributing to the above summation
describes an ordered factorization of $D^{\underline{b}}=D_1\cdots D_\ell$ into $k$ non-trivial coprime parts.
Denote by $ N_{k,\ell} $
the number of such length $k$ factorizations.
Using that $D_\ell$ is a factor of $D^{\underline{b}_j}$ for a unique $j$ we may uniquely associate to each length $k$ ordered factorization of $D_1\cdots D_\ell$ an ordered factorization of $D_1\cdots D_{\ell-1}$ of either length $k$ or length $k-1$ as follows:
\begin{itemize}
\item If $D^{\underline{b}_j} \neq D_\ell$ then replace $\underline{b}_j$ by $\underline{b}_j'$ where $D^{\underline{b}_j} = D^{\underline{b}_j'}D_\ell$. This gives a length $k$ factorization.
\item If $D^{\underline{b}_j} = D_\ell$ we omit $\underline{b}_j$ from the factorization entirely, and shift down the indices on $\underline{b}_i$ for $i>j$. This gives a length $k-1$  factorization.
\end{itemize}
As we run over all the ordered factorizations of $D_1\cdots D_{\ell}$ each length $k$ and each length $k-1$ ordered factorization of $D_1\cdots D_{\ell-1}$ occurs exactly $k$ times. We thus obtain a recurrence relation $ N_{k,\ell} = kN_{k,\ell-1} + k N_{k-1,\ell-1}$ and a straightforward computation yields that:
\[ \sum_{k\ge 0} (-1)^{k+1}N_{k,\ell} = - \sum_{k\ge 0} (-1)^{k+1}N_{k,\ell-1}.\]
The claim now follows by an induction on $\ell = \abs{\underline{b}}$.
\end{proof}

\begin{prop}\label{prop:constDNMF}
When $\underline{b}$ is NMF and $\abs{\hat{\underline{b}}} \ge 1$ then:
\[  \sum_{k\ge 0} (-1)^{k+1}  \underset{\sum {\underline{b}_j} = \underline{b}}{\sum_{(\underline{b}_1,\ldots,\underline{b}_k)}} 1 = 0 \]
where the sum is taken as above.
\end{prop}
\begin{proof}
Every ordered factorization of $D^{\underline{b}}$ into $k$ non-trivial coprime parts where only the last one is NMF induces an ordered factorization of  $D^{\hat{\underline{b}}}$ into either $k-1$ non-trivial coprime parts or $k$  non-trivial coprime parts.
Each factorization of $D^{\hat{\underline{b}}}$ arises in exactly two ways. It follows that:
\[  \sum_{k\ge 0} (-1)^{k+1}  \underset{\sum {\underline{b}_j} = \underline{b}}       {\sum_{(\underline{b}_1,\ldots,\underline{b}_k)}} 1 = 
     \sum_{k\ge 0} (-1)^{k+1}  \underset{\sum {\underline{b}_j} = \hat{\underline{b}}}{\sum_{(\underline{b}_1,\ldots,\underline{b}_k)}} 1 - 
     \sum_{k\ge 0} (-1)^{k+1}  \underset{\sum {\underline{b}_j} = \hat{\underline{b}}}{\sum_{(\underline{b}_1,\ldots,\underline{b}_k)}} 1 = 0\]
which gives the desired result.
\end{proof}
The constants $\constD_{\underline{b}}$ shall be used in Corollaries \ref{cor:secondaryinduction}, \ref{cor:leprim/imprim} and  \ref{cor:main}.

As an example, by considering the different ordered decompositions of $(1,1,1)$, for instance:
\[ (1,1,1) \quad (1,1,0) + (0,0,1)   \quad  (0,0,1) +(1,1,0) \quad   (1,0,1) +(0,1,0)  \quad \ldots \]
including also the $6$ permutations of $ (1,0,0) + (0,1,0) + (0,1,0)$,
we see that:
\[ \constD_{(1,1,1) } = \constC_{(1,1,1)} - 6\constC_{(1,1)}\constC_{(1)} + 6\constC_{(1)}^3 = \frac{1}{8}. \]
We can also compute that:
\[ \constD_{(0)} = -1 \qquad \constD_{(1)} = \constC_{(1)} = \frac{1}{2} \qquad \constD_{(1,1)} = \constC_{(1,1)} -2 \constC_{(1)}^2 = -\frac{1}{4}  \]
\[ \constD_{(2)} = \constC_{(2)} = \frac{1}{12} \qquad \constD_{(2,1)} =  \constC_{(2,1)} -  \constC_{(2)}\constC_{(1)} = 0 \qquad \constD_{(3)} =  \constC_{(3)} = 0. \]
\end{enumerate}
\end{nota}

\begin{prop}\label{prop:LogChernvsChern}
Let $X$ be a smooth projective variety and $\Delta = \cup \{D_i\}$ be a collection of smooth divisors with simple normal crossings on $X$.
We have a relation:
\[ \chern_i(\Omega^1_{X}) = \sum_j (-1)^{i-j}\chern_j(\Omega^1_{X}(\log \Delta))\Delta_{i-j}. \]
Recall $\Delta_k$ is the $k$-th elementary symmetric polynomial in the irreducible components of the boundary of $X$.
This can also be expressed as:
\[ \chern(\Omega^1_{X})  = \chern(\Omega^1_{X}(\log))\prod_{D_i}(1-{D_i}). \]
\end{prop}
\begin{proof}
We follow essentially an argument for an analogous result from \cite[Prop. 1.2]{Tsushima}.
We have the following two exact sequences:
\[ \xymatrix{
0 \ar[r]&  \Omega^1_{X} \ar[r]&\Omega^1_{X}(\log \Delta) \ar[r]& \oplus \cO_{D_i} \ar[r]& 0\\
0 \ar[r]& \cO_{X}(-D_i) \ar[r]& \cO_{X} \ar[r]& \cO_{D_i} \ar[r]& 0. 
}\]
The first of which essentially defines $\Omega^1_{\overline{X}}(\log \Delta)$.

By the multiplicativity of the total Chern class we obtain:
\[ \chern(\Omega_X^1) = \chern(\Omega_X^1(\log \Delta)) \prod_{D_i} (1-D_i). \qedhere\]
\end{proof}

\begin{prop}\label{prop:CHonBDY}
Logarithmic Chern classes restrict to the boundary. 
That is, let $X$ be a smooth projective variety and $\Delta = \cup \{D_i\}$ be a collection of smooth divisors with simple normal crossings on $X$.
Suppose $D \in \Delta$ is a fixed irreducible divisor then:
 \[ \chern^\alpha(\Omega_{X}^1(\log \Delta)) \cdot D = \chern^\alpha(\Omega_{D}^1(\log \Delta')). \]
This equality should be interpreted as an equality on $D$.
\end{prop}
\begin{proof}
The result is analogous to \cite[Lem. 5.1]{Tsushima}, this proof was suggested by the referee.

By Proposition \ref{prop:LogChernvsChern} we have:
\[ \chern(\Omega_{X}^1(\log \Delta)) = \chern(\Omega_{X}^1) \frac1{(1-D)}\prod_{D_i\neq D} \frac1{(1-D_i)}. \]
As $ \chern(\Omega_{X}^1) \frac1{(1-D)}$ restricts to $ \chern(\Omega_{D}^1)$ on $D$ the right hand side of the above expression restricts to:
\[ \chern(\Omega_{D}^1)\prod_{D_i\neq D} \frac1{(1-D_i')} \]
which in turn equals $  \chern(\Omega_{D}^1(\log \Delta')) $ by Proposition \ref{prop:LogChernvsChern}.
As the Chern classes agree, so to do their products.
\end{proof}

\begin{nota}
Consider $\pi: X \rightarrow Y$ a ramified covering.
For $Z \subset X$ irreducible we shall denote by $e_Z$ the ramification degree of $\pi$ at $Z$ as it is defined in \cite[Ex. 4.3.4]{FultonIntersection1}.

We note that in the context of smooth varieties by \cite[Prop. 7.1]{FultonIntersection1} we can compute the ramification degree as:
\[ e_Z = {\rm length}(\cO_{X,Z} \otimes_{\cO_Y} \cO_{Y,\pi(Z)}/J_{\pi(Z)}). \]
In the expression above, $J_{\pi(Z)}$ is the ideal associated to $\pi(Z)$ and the length is that of the ring as a  module over itself.
\end{nota}

The following  proposition is well known (see for example \cite[Ex. 4.3.7]{FultonIntersection1}), though we will not make direct use of it the statement motivates our understanding of ramification.
\begin{prop}\label{prop:sumramdeg}
Let $X$ and $Y$ be smooth projective varieties.
Consider $\pi: X \rightarrow Y$ a potentially ramified finite covering of degree $\mu$.
For any $Z' \subset Y$ irreducible, if we decompose $\pi^{-1}(Z') = \cup_i Z_i$ into irreducible components then:
\[ \sum_i \mu_{Z_i}e_{Z_i} = \mu. \]
where $\mu_{Z_i}$ is the degree of $\pi|_{Z_i}$.
\end{prop}

\begin{nota}\label{not:ERA}
Fix a ramified covering $\pi : X \rightarrow Y$ of smooth projective varieties of dimension $n$.

The collection of reduced irreducible components of the branch locus shall be denoted $\Delta(B)$, and we shall denote monomial exponents for the branch locus by $\underline{b}$ and write $B^{\underline{b}}$ for the associated equivalence class of cycle.

The collection of reduced irreducible components of the ramification locus shall be denoted $\Delta(R)$, and we shall  denote monomial exponents for the ramification locus by $\underline{a}$ and write $R^{\underline{a}}$ for the associated equivalence class of cycle.
Recall that  $\Delta(R) = \pi^{-1}(\Delta(B))$ includes all components $R_j$ in $\pi^{-1}(B_i)$ even those which may not themselves be ramified. 

For an irreducible component $R_i$ then it is clear $\pi(R_i) = B_j$ for a unique $j$.
Given a pair of monomial exponents $\underline{a}$ and $\underline{b}$ we shall say $\pi(\underline{a}) = \underline{b}$ if for each $j$ we have:
\[ b_j = \sum_{\pi(R_i) = B_j} a_i . \]
We shall denote by:
\[ E_{R^{\underline{a}}} =  \prod_i (e_{R_i})^{a_i}  \]
the product of the ramification degrees. This notation is justified by Lemma \ref{lem:cute2} which says that when $\underline{a}$ is MF then $E_{R^{\underline{a}}}$ is the ramification degree of each irreducible component of $R^{\underline{a}}$.
\end{nota}

\begin{lemma}\label{lem:cute1}
Consider $\pi: X \rightarrow Y$ a potentially ramified finite map between smooth projective varieties.
Suppose $D_1$ and $D_2$ are two (reduced irreducible) divisors on $X$ which meet with simple normal crossings and that $\pi(D_1) = \pi(D_2) = D$ is smooth.
Let $Z$ be a (reduced) irreducible component of $D_1\cap D_2$.
Then there is a component $R$ of the ramification locus of $\pi$ such that $Z \subset R$ and $Z \not\in \{ D_1,D_2 \}$.
In particular the collection $ \{ D_1,D_2, R \}$ does not have simple normal crossings.
\end{lemma}
\begin{proof}
We will consider the completed local rings at $Z$ and $\pi(Z)$. By the Cohen structure theorem (for regular complete local rings, see \cite[Tag 0323 Lemma 10.154.10]{stacks-project}) these are power series rings over the coordinate ring, we can thus write these in the form:
\[ K(Z)[[s_1,s_2]] \quad \text{and}\quad K( \pi(Z))[[t_1,t_2]] \]
where $s_i$ is the local coordinate defining $D_i$ on $X$ near $Z$, the coordinate $t_1$ defines $D$ on $Y$ near $\pi(Z)$ and $t_2$ is any other local coordinate defining a divisor which meets $D$ transversely at $\pi(Z)$.

By the assumption that $\pi(D_1)=\pi(D_2)=D$ we have that we can choose our coordinates $s_1$ and $s_2$ so that $\pi^\ast(t_1) = us_1^{a_1}s_2^{a_2}$ with $a_1,a_2\ge1$ and $u\in K(Z)^\times$.
By the assumption that the map is finite we have that $s_1,s_2\not\vert \pi^\ast(t_2)$, moreover $\pi^\ast(t_2)$ vanishes at $Z$ and thus  $\pi^\ast(t_2)$ has trivial constant term.
It follows that 
\[ \pi^\ast(t_2) = v_1s_1^{b_1} + v_2s_2^{b_2} + (\text{other terms (not including those monomials)})\]
with $b_1,b_2 \ge 1$ and $v_1,v_2 \in  K(Z)^\times$.

We can understand the ramification locus near $Z$ by way of the Jacobian condition.
The Jacobian is precisely:
\[ a_1us_1^{a_1-1}s_2^{a_2} \left( b_2v_2s_2^{b_2-1}+  \frac{\del (\text{other terms})}{\del s_2}\right) + a_2us_1^{a_1}s_2^{a_2-1} \left( b_1v_1s_1^{b_1-1}  + \frac{\del (\text{other terms})}{\del s_1} \right). \]
As the expressions $s_2 \frac{\del (s_1^{\ell_1}s_2^{\ell_2})}{\del s_2}$ and $s_1 \frac{\del (s_1^{\ell_1}s_2^{\ell_2})}{\del s_1}$ both have the same monomial type, namely $s_1^{\ell_1}s_2^{\ell_2}$, as the starting monomial we find  that we may rewrite the Jacobian above as:
\[ us_1^{a_1-1}s_2^{a_2-1}( a_2b_1v_1s_1^{b_1}  + a_1b_2v_2s_2^{b_2} + (\text{other terms (not including those monomials)})). \]
The term $(a_2b_1v_1s_1^{b_1}+  a_1b_2v_2s_2^{b_2}   + (\text{other terms}))$ vanishes at $Z$ and is not divisible by $s_1$ or $s_2$ and thus defines at least one component of the ramification locus that pass through $Z$ which is not equal to $D_1$ or $D_2$.
\end{proof}

\begin{lemma}\label{lem:cute2}
Consider $\pi: X \rightarrow Y$ a potentially ramified finite map between smooth projective varieties. 
Let $\Delta(B)$ be the collection of irreducible components of the branch locus (on $Y$) and $\Delta(R) = \pi^{-1}(\Delta(B))$ be the collection of (reduced) irreducible components of the ramification locus (on $X$). Suppose $\Delta(B)$ and  $\Delta(R)$ have simple normal crossings.
If $R_1,\ldots R_\ell \in \Delta(R)$ are distinct and if $Z$ is a (reduced) irreducible component of $\cap_i R_i$ then:
\[ e_Z = \prod_i e_{R_i}. \]
\end{lemma}
\begin{proof}
We will consider the completed local ring at $Z$ and $\pi(Z)$.
The completed local rings at the generic points are of the form:
\[ K(Z)[[s_1,\ldots,s_{\ell}]] \quad \text{and}\quad K( \pi(Z))[[t_1,\ldots,t_{\ell}]] \]
where $s_i$ is a local parameter defining $R_i$, and $t_i$ a local parameter defining $B_i = \pi(R_i)$. That $\pi(R_i)$ are all distinct follows from  Lemma \ref{lem:cute1}.
It follows from this setup that we may choose the local coordinate $s_i$ so that $\pi^\ast(t_i) = u_is_i^{a_i}$  with $a_i \ge 1$ and $u_i\in K(Z)^\times$.
The claim now follows from a direct computations of lengths.
In particular $e_{R_i} = a_i$ and $e_Z = \prod_i a_i$.
\end{proof}

\begin{lemma}\label{lem:cute3}
Consider $\pi: X \rightarrow Y$ a potentially ramified finite map between smooth projective varieties. 
Let $\Delta(B)$ be the collection of irreducible components of the branch locus (on $Y$) and $\Delta(R) = \pi^{-1}(\Delta(B))$ be the collection of (reduced) irreducible components of the ramification locus (on $X$). Suppose $\Delta(B)$ and  $\Delta(R)$ have simple normal crossings.
\begin{enumerate}
\item If $\pi(\underline{a}) = \underline{b}$, and the monomial type of $\underline{a}$ and $\underline{b}$ are not the same then:
\[ R^{\underline{a}} = 0. \]

\item If $\pi(\underline{a}) = \underline{b}$, and the monomial type of $\underline{a}$ and $\underline{b}$ are the same then in the formal expansion:
\[  \pi^\ast(B^{\underline{b}}) = \prod_i \pi^{\ast}(B_i)^{b_i} = \prod_i \left( \sum_{\pi(R_j) = B_i} e_{R_j}R_j \right)^{b_i} = \sum_{\pi(\underline{a})=\underline{b}} x_{\underline{a}}R^{\underline{a}} \]
the coefficient $x_{\underline{a}}$ of $R^{\underline{b}}$ is $E_{R^{\underline{a}}}$.

\item We have the following identity in the Chow ring:
 \[ \pi^\ast(B^{\underline{b}}) = \sum_{\pi(\underline{a}) = \underline{b}} E_{R^{\underline{a}}} R^{\underline{a}}. \]
\end{enumerate}
\end{lemma}
\begin{proof}
The first statement follows immediately from Lemma \ref{lem:cute1}, in particular if the monomial exponents are not the same then the expression $R^{\underline{a}}$ involves intersecting two components which map to the same $B_i$. If these two components do not have trivial intersection than the ramification locus does not have simple normal crossings.

The second statement is a straightforward check and indeed is a basic property of multinomial coefficients.

The third statement then combines the previous two by observing that $R^{\underline{a}} = 0$ whenever the coefficient of $R^{\underline{a}}$ is not $E_{R^{\underline{a}}}$.
\end{proof}

\begin{thm}\label{thm:pullbacklog}
Logarithmic Chern classes respect pullbacks through ramified covers.
That is, let $X$ and $Y$ be smooth projective varieties of dimension $n$.
Consider $\pi: X \rightarrow Y$ a potentially ramified finite covering.
Let $\Delta(B)$ be the collection of irreducible components of the branch locus (on $Y$) and $\Delta(R) = \pi^{-1}(\Delta(B))$ be the collection of (reduced) irreducible components of the ramification locus (on $X$). 
Suppose that $\Delta(R)$ and $\Delta(B)$ consist of simple normal crossing divisors.
Then:
\[ \pi^\ast( \Omega_Y(\log \Delta(B)) =  \Omega_X(\log \Delta(R))). \]
Recall that $\Delta(R)$ includes even those irreducible components of $\pi^{-1}(B)$ which are not themselves ramified.
\end{thm}
\begin{proof}
The claim can be checked locally on $Y$.

Suppose $x_1,\ldots,x_n$ are a local system of coordinates at some point $\underline{x}$ of $X$, and $y_1,\ldots,y_n$ are a local system of coordinates near $\underline{y} = \pi(\underline{x})$.
We may suppose that $y_1,\ldots,y_\ell$ define the branch locus of $\pi$ near $\underline{y}$ and further that $\pi^\ast(y_i)  = x_i^{a_i}$ so that $x_1,\ldots,x_\ell$ define the ramification locus of $\pi$ near $\underline{x}$ (see proof of Lemma \ref{lem:cute2}). 
Set $\epsilon_i = 1$ if $i \leq \ell$ and $0$ otherwise.
Then the bundle $\Omega_Y(\log \Delta(R))$ has a basis of sections near $\underline{y}$ given by:
\[ \frac{dy_1}{y_1^{\epsilon_1}} ,\ldots, \frac{dy_n}{y_n^{\epsilon_n}}. \]
By the choice of $\epsilon_i$ we find that for all $i$:
\[ \frac{d (\pi^\ast y)}{ (\pi^\ast y)^{\epsilon_i}} = \frac{d (x_i^{a_i})}{ x_i^{a_i \epsilon_i}} = a_i\frac{dx_i}{x_i^{\epsilon_i}}  \]
we find that $\pi^\ast(\Omega_Y(\log \Delta(R)))$ has a basis of sections near $\underline{x}$:
\[ \frac{dx_1}{x_1^{\epsilon_1}} ,\ldots, \frac{dx_n}{x_n^{\epsilon_n}}. \]
This precisely agrees with the bundle $\Omega_X(\log \Delta(B))$ near $\underline{x}$.
\end{proof}

\section{Logarithmic Euler Characteristic}\label{sec:logeuler}

Aside from its present application to a Riemann-Hurwitz formula, the following definition is motivated in part by its appearance in Mumford's work in \cite[Cor. 3.5]{Mumford_Proportionality}.
\begin{df}
Let $X$ be a smooth projective variety and $\Delta$ be a collection of smooth divisors with simple normal crossings on $X$.
We define the logarithmic Euler characteristic of a sheaf $\cF$ on $X$ with respect to the boundary $\Delta$ to be:
\[ \chi(X,\Delta,\cF) = Q_n( \chernchar(\cF);  \chern_1(\Omega^1_{\overline{X}}(\log \Delta)),\ldots, \chern_n(\Omega^1_{\overline{X}}(\log \Delta))). \]
\end{df}

Though it is not a priori clear what use this definition can have, the following theorem shows that it in some sense behaves better than the standard Euler characteristic.

\begin{thm}\label{thm:pullbacklogchern}
Let $X$ and $Y$ be smooth projective varieties.
Consider $\pi: X \rightarrow Y$ a potentially ramified finite covering of degree $\mu$. 
Let $\Delta(B)$ be the collection of irreducible components of the branch locus (on $Y$) and $\Delta(R) = \pi^{-1}(\Delta(B))$ be the collection of (reduced) irreducible components of the ramification locus (on $X$). Suppose $\Delta(B)$ and  $\Delta(R)$ have simple normal crossings.

Let $\cF$ be any coherent sheaf on $Y$ then:
\[ \chi(X,\Delta(R),\pi^\ast(\cF)) = \mu \cdot \chi(Y,\Delta(B),\cF). \]
\end{thm}
\begin{proof}
By Theorem \ref{thm:pullbacklog} (and functoriality) we have that:
\[ \pi^\ast(\chernchar(\cF){\rm Todd}(\Omega^1_Y(\log \Delta(B)))) = \chernchar(\pi^{\ast}(\cF)){\rm Todd}(\Omega^1_X(\log \Delta(R))). \]
The result then follows by recalling that the effect of pullback on the degree of a class is to multiply by $\mu$.
\end{proof}

\section{Logarithmic Euler Characteristic vs The Euler Characteristic}\label{sec:logveuler}

The key to obtaining our results is the following comparison between the usual Euler characteristic and the logarithmic Euler characteristic we have just defined.

\begin{thm}\label{thm:eulvslogeul}
Let $X$ be a smooth projective variety and let $\cF$ be any coherent sheaf on $X$.
Suppose $\Delta$ is a collection of smooth divisors with simple normal crossings on $X$.
Then
\[
 \chi(X,\cF) - \chi(X,\Delta,\cF)  =  \sum_{\abs{\underline{b}}\ge 1} (-1)^{\abs{\underline{b}}} \constC_{\underline{b}} D^{\underline{b}} Q_{n-\abs{\underline{b}}}(\chernchar(\cF); \chern_1(\Omega_X^1(\log \Delta)),\ldots, \chern_{n-\abs{\underline{b}}}(\Omega_X^1(\log\Delta))) .
\]
The notation $D^{\underline{b}}$ is introduced in \ref{not:1}.(9), the polynomial $Q$ is defined in Theorem \ref{thm:RR}, the constants $\constC_{\underline{b}}$ are introduced in \ref{not:2}.(1).
\end{thm}
\begin{proof}
Recall that by Proposition \ref{prop:multQ} we have:
\[ Q_n(\chernchar(\underline{x}); y_1,\ldots,y_n) = \sum_{\ell + m = n} Q_\ell(\chernchar(\underline{x}); u_1,\ldots,u_\ell) Q_m(1; v_1,\ldots,v_m). \]
In this context if we set $x_i = \chern_i(\cF)$, $u_i = \chern_i(\Omega_X^1(\log \Delta))$, and $v_i = (-1)^i\Delta_i$ then by Proposition \ref{prop:LogChernvsChern} we have in the setting of Proposition \ref{prop:multQ} that $y_i = \chern_i(\Omega_X^1)$ and it follows that we can rewrite $ Q_n(\chernchar(\cF); \chern_1(\Omega_X^1),\ldots, \chern_n(\Omega_X^1)) $ as being equal to:
\[ \sum_{\ell+m=n} Q_\ell(\chernchar(\cF); \chern_1(\Omega_X^1 (\log\Delta))\ldots, \chern_{\ell}(\Omega_X^1(\log\Delta)))  Q_m(1;  (-1)^1\Delta_1, (-1)^2\Delta_2,\ldots, (-1)^m\Delta_m). \]
The result then follows from the observation that:
\[ Q_m(1;  (-1)^1\Delta_1, (-1)^2\Delta_2,\ldots, (-1)^m\Delta_m) = \sum_{\abs{\underline{b}}=m} (-1)^{\abs{\underline{b}}} \constC_{\underline{b}} D^{\underline{b}}. \qedhere\]
\end{proof}

\begin{nota}\label{not:logeuler}
If $\underline{a}$ is multiplicity free, so that $D^{\underline{a}}$ has no self intersections, then we may write
\[ D^{\underline{a}} = \sum_j x_j [C_j] \quad\text{ where }\quad \left(\underset{a_i\neq0}\cap\; D_i\right)_{red} = \underset{j}\cup \;C_j   \]
in this setting we interpret $\chi(D^{\underline{a}},\Delta',\cF|_{D^{\underline{a}}}) $ to mean: 
\[  \chi(D^{\underline{a}},\Delta',\cF|_{D^{\underline{a}}}) = \sum_i m_i\chi(C_i,\Delta',\cF|_{C_i}) \]
the weighted sum of the logarithmic Euler characteristics of the irreducible components of $D^{\underline{a}}$, the weights being precisely the intersection multiplicities. We interpret $\chi(D^{\underline{a}},\cF|_{D^{\underline{a}}}) $ similarly.
Both of these expressions most naturally live on the disjoint unions of irreducible components of $D^{\underline{a}}$. 
Note that in the context of simple normal crossings the intersection will already be reduced and the multiplicities, $m_i$, will all be $1$.

\medskip
When $\underline{a}$ is MF we have by Proposition \ref{prop:CHonBDY} that 
\[  \chi(D^{\underline{a}},\Delta',\cF|_{D^{\underline{a}}}) = D^{\underline{a}} Q_{n-\abs{\alpha}}(\chernchar(\cF); \chern_1(\Omega_X^1(\log\Delta)),\ldots,\chern_{n-\abs{\alpha}}(\Omega_X^1(\log\Delta))) \]
when this expression is viewed as an equality on the disjoint union of the irreducible components of $D^{\underline{a}}$.

\medskip
By an abuse of notation we shall extend this to the case where there may be self intersections and denote by:
\[  \chi(D^{\underline{a}},\Delta',\cF|_{D^{\underline{a}}}) = D^{\underline{a}} Q_{n-\abs{\alpha}}(\chernchar(\cF); \chern_1(\Omega_X^1(\log\Delta)),\ldots,\chern_{n-\abs{\alpha}}(\Omega_X^1(\log\Delta))) \]
even when $a_i$ are potentially greater than $1$ so that we may interpret $ \chi(D^{\underline{a}},\Delta',\cF|_{D^{\underline{a}}}) $ as an object on $X$. This interpretation is compatible with the interpretation as a push-forward whenever $\underline{a}$ is MF.
\end{nota}

\begin{cor}\label{cor:secondaryinduction}
With the same notation as in the Theorem,
if the irreducible components of $\Delta$ have trivial self intersection, then:
\[ \chi(X,\cF)- \chi(X,\Delta,\cF) = \sum_{\abs{\underline{b}}\ge 1} (-1)^{\abs{\underline{b}}}\constD_{\underline{b}} \chi(D^{\underline{b}},\cF|_{D^{\underline{b}}}). \]
The notation $D^{\underline{b}}$ is introduced in \ref{not:1}.(9), the constants $\constD_{\underline{b}}$ are introduced in \ref{not:2}.(4).
\end{cor}
\begin{proof}
In the above notation Theorem \ref{thm:eulvslogeul} gives us that:
\[  \chi(X,\cF)- \chi(X,\Delta,\cF) = \sum_{\abs{\underline{b}}\ge 1} (-1)^{\abs{\underline{b}}}\constC_{\underline{b}} \chi(D^{\underline{b}},\Delta',\cF|_{D^{\underline{b}}}). \]
As the same result allows us to compute  $\chi(D^{\underline{b}},\Delta',\cF|_{D^{\underline{b}}}) - \chi(D^{\underline{b}},\cF|_{D^{\underline{b}}})$ whenever $\underline{b}$ is MF a recursive process will allow us to write:
\[ \chi(X,\cF) - \chi(X,\Delta,\cF)  = \sum_{\abs{\underline{b}}\ge 1} e_{\underline{b}} \chi(D^{\underline{b}},\cF|_{D^{\underline{b}}}). \]
We must only show that $e_{\underline{b}} = (-1)^{\underline{b}}\constD_{\underline{b}}$

The coefficient of $\chi(D^{\underline{b}},\cF|_{D^{\underline{b}}})$ can be computed by explicitly writing out the result of the recursive process.
The process will yield a sequence of formulas, indexed by $\ell$, of the form:
\begin{align*}
\chi(X,\cF) - \chi(X,\Delta,\cF)  = \sum_{k=1}^{\ell-1} (-1)^{k+1} \sum_{(\underline{b}_1,\ldots,\underline{b}_k)} \left(\prod_{j=1}^k (-1)^{\abs{\underline{b}_j}}\constC_{\underline{b}_j}\right) \chi(D^{\sum_{j=1}^k \underline{b}_j}, \cF|_{D^{\sum_{j=1}^k \underline{b}_j}}) \\
+ (-1)^{\ell+1} \sum_{(\underline{b}_1,\ldots,\underline{b}_{\ell})} \left(\prod_{j=1}^{\ell}  (-1)^{\abs{\underline{b}_j}}\constC_{\underline{b}_j}\right) \chi(D^{\sum_{j=1}^k \underline{b}_j},\Delta' , \cF|_{D^{\sum_{j=1}^{\ell+1} \underline{b}_j}})
\end{align*}
In the summations the elements of the tuples $(\underline{b}_1,\ldots,\underline{b}_k)$ always have disjoint support and $\abs{\underline{b}_j} \ge 1$.
We note that in the context of this corollary we need never consider any terms where $\underline{b} = \sum_{j=1}^{\ell} \underline{b}_j$ is NMF as for each such term we have $D^{\underline{b}}$ vanishes.

The base case of the induction, the case $\ell=1$, is precisely the statement of Theorem \ref{thm:eulvslogeul}.

The formula for $\ell+1$ is obtained from that for $\ell$ by simply expanding every term:
\[ \chi(D^{\sum_{j=1}^k \underline{b}_j},\Delta' , \cF|_{D^{\sum_{j=1}^{\ell} \underline{b}_j}}) = \chi(D^{\sum_{j=1}^{\ell} \underline{b}_j}, \cF|_{D^{\sum_{j=1}^{\ell} \underline{b}_j}}) - \sum_{\underline{c}} (-1)^{\abs{\underline{c}}}\constC_{\underline{c}}\chi(D^{\underline{c}+\sum_{j=1}^k \underline{b}_j},\Delta', \cF|_{D^{\underline{c}+\sum_{j=1}^{\ell+1} \underline{b}_j}}) \]
with each term $\underline{c}$ in the summation avoiding the support of $\sum_{j=1}^k \underline{b}_j$.
This recursion terminates as soon as $\ell > n$ because then $D^{\underline{b}}$ is an intersection of more than $n$ divisors, hence empty.

By regrouping terms on $\chi(D^{\underline{b}},\cF|_{\underline{b}})$ we find that the coefficient of this term is precisely \[ (-1)^{\abs{\underline{b}}}\constD_{\underline{b}} = (-1)^{\abs{\underline{b}}} \sum_{k\ge 0} \underset{\sum {\underline{b}_j} = \underline{b}}{\sum_{(\underline{b}_1,\ldots,\underline{b}_k)}} \left(\prod_{j=1}^k \constC_{\underline{b}_j}\right). \]
In the summation we consider only terms with all $\abs{\underline{b}_j} \ge 1$ and where in the tuple $(\underline{b}_1,\ldots,\underline{b}_k)$ all of $\underline{b}_j$ have disjoint support. For $k$ sufficiently large the inner sum is an empty sum.
\end{proof}

\begin{rmk}
The proofs of the above theorem and corollary work formally when we replace $Q(x_1,\ldots,x_n; y_1,\ldots,y_n)$ by any other polynomial which is a multiplicative sequence in the $y_i$ with respect to products of varieties and such that the $x_j$ are `functorial' with respect to restriction.

We should also note that in light of Propositions \ref{prop:constC} and \ref{prop:constD} the coefficient $(-1)^{\abs{\underline{b}}}\constD_{\underline{b}}$ can be rewritten as $\constC_{(1)}^{\abs{\underline{b}}}$ whenever $\underline{b}$ is MF (as in the Corollary above or below). Also, by Proposition \ref{prop:constDNMF} the constants $\constD_{\underline{b}}$ in the Corollary below are typically $0$ when $\underline{b}$ is NMF.
\end{rmk}

\begin{cor}\label{cor:leprim/imprim}
With the same notation as in the Theorem, we have:
\begin{align*}
 \chi(X,\cF)  - \chi(X,\Delta,\cF) = \underset{\abs{\underline{b}} \ge 1}{\sum_{\underline{b}\;{\rm MF}}} &(-1)^{\abs{\underline{b}}}\constD_{\underline{b}} \chi(D^{\underline{b}},\cF|_{D^{\underline{b}}})
 \\&+   \sum_{\underline{b}\;{\rm NMF}} (-1)^{\abs{\underline{b}}}{\constD}_{\underline{b}} \chi(D^{\underline{b}},\Delta',\cF|_{D^{\underline{b}}}).
\end{align*}
The notation $D^{\underline{b}}$ is introduced in \ref{not:1}.(9), the terminology MF and NMF is from \ref{not:2}.(2), the constants $\constD_{\underline{b}}$ are introduced in \ref{not:2}.(4).
\end{cor}
\begin{proof}
The argument is the same as above, except rather than being able to completely ignore any NMF term which may appear, we simply include their contribution separately. The constant ${\constD}_{\underline{b}}$ is defined precisely so as to count the appropriate weighted count of the number of possible factorizations of $D^{\underline{b}}$ in which terms have disjoint support and only the final term is potentially NMF.
\end{proof}

\section{Riemann-Hurwitz}\label{sec:rh}

In this section we establish our main result.

\begin{thm}\label{thm:main}
Consider $\pi: X \rightarrow Y$ a potentially ramified finite covering of degree $\mu$ between smooth projective varieties of dimension $n$. 
Let $\Delta(B)$ be the collection of irreducible components of the branch locus (on $Y$) and $\Delta(R) = \pi^{-1}(\Delta(B))$ be the collection of (reduced) irreducible components of the ramification locus (on $X$).
Let $\cF$ be any coherent sheaf on $Y$.
Suppose that $\Delta(R)$ and $\Delta(B)$ consist of simple normal crossing divisors.
We have that:
\[ \chi(X, \pi^\ast( \cF)) - \mu\cdot \chi(Y,\cF) = \sum_{\underline{a}} (-1)^{\abs{\underline{a}}}\constC_{\underline{a}}(E_{R^{\underline{a}}} - 1)\chi(R^{\underline{a}},\Delta(R)', \pi^\ast( \cF)).
 \]
The notation $D^{\underline{b}}$ is introduced in \ref{not:1}.(9), the constants $\constC_{\underline{b}}$ are introduced in \ref{not:2}.(1), the notation $E_{R^{\underline{a}}}$ is from \ref{not:ERA}, and the notation $\chi(R^{\underline{a}},\Delta(R)', \pi^\ast( \cF))$ is from \ref{not:logeuler}.
\end{thm}
\begin{proof}
Firstly, by Theorem \ref{thm:eulvslogeul} we have:
\begin{align*}
\chi(X, \pi^\ast( \cF)) - \mu\cdot \chi(Y,\cF) =  \sum_{\abs{\underline{a}}\ge 0} &(-1)^{\abs{\underline{a}}}\constC_{\underline{a}} \chi(R^{\underline{a}},\Delta(R)',\pi^\ast( \cF)|_{R^{\underline{a}}})\\ &- \mu\left( \sum_{\abs{\underline{b}}\ge 0} (-1)^{\abs{\underline{b}}}\constC_{\underline{b}} \chi(B^{\underline{b}},\Delta',\cF|_{B^{\underline{b}}})  \right). 
\end{align*}
Next, by Theorem \ref{thm:pullbacklogchern} we have that:
\[ \chi(X,  \Delta(R), \pi^\ast (\cF))) = \mu\cdot \chi(Y,  \Delta(B), \cF)) \]
so that these terms cancel out in the above expression.

With the remaining terms we can naturally group together those terms involving $\underline{a}$ and those with $\pi(\underline{a}) = \underline{b}$ in the summation above.
The error term arising from $\underline{a}$ in the expansion is:
\[  (-1)^{\abs{\underline{a}}}(\mu \constC_{\underline{b}} \chi(B^{\underline{b}},  \Delta(B)', \cF)) - \sum_{\pi(\underline{a}) = \underline{b}} \constC_{\underline{a}}\chi(R^{\underline{a}},  \Delta(R)', \pi^\ast (\cF))). \]
Next, we observe that:
\begin{align*} \mu\cdot \chi(B^{\underline{b}},  \Delta(B)', \cF)) &= \mu (B^{\underline{b}} Q_n(\chernchar(\cF);\chern_1(\Omega_Y^i(\log\Delta(B))),\ldots,\chern_n(\Omega_Y^i(\log\Delta(B)))) \\
                                       &= \pi^\ast (B^{\underline{b}} Q_n(\chernchar(\cF);\chern_1(\Omega_X^i(\log\Delta(B))),\ldots,\chern_n(\Omega_X^i(\log\Delta(B)))) \\
                                       &=  \pi^\ast(B^{\underline{b}})  Q_n(\chernchar(\pi^\ast( \cF));\chern_1(\Omega_X^i(\log\Delta(R))),\ldots,\chern_n(\Omega_X^i(\log\Delta(R)))).
\end{align*}
By Lemma \ref{lem:cute3} we have that
\[  \pi^\ast(B^{\underline{b}}) = \sum_{\pi(\underline{a}) = \underline{b}} E_{R^{\underline{a}}} R^{\underline{a}} \]
and so we obtain:
\[  \mu \cdot \chi(B^{\underline{b}},  \Delta(B)', \cF))  =  \sum_{\pi(\underline{a}) = \underline{b}} E_{R^{\underline{a}}}\chi(R^{\underline{a}},  \Delta(R)', \pi^\ast(\cF)). \]
Grouping the terms on $\underline{a}$ we now immediately see that the contribution from the $\underline{a}$ terms is:
\[  (-1)^{\abs{\underline{a}}}\constC_{\underline{a}}(E_{R^{\underline{a}}} - 1)\chi(R^{\underline{a}},\Delta(R)', \pi^\ast (\cF))). \]
Collecting these over all $\underline{a}$ we obtain the theorem.
\end{proof}

The coefficients in the following corollary can be rewritten using Propositions \ref{prop:simpleMF} and \ref{prop:simpleNMF}.
\begin{cor}\label{cor:main}
Consider $\pi: X \rightarrow Y$ a potentially ramified finite covering of degree $\mu$ between smooth projective varieties of dimension $n$. 
Let $\Delta(B)$ be the collection of irreducible components of the branch locus (on $Y$) and $\Delta(R) = \pi^{-1}(\Delta(B))$ be the collection of (reduced) irreducible components of the ramification locus (on $X$).
Let $\cF$ be any coherent sheaf on $Y$.
Suppose that $\Delta(R)$ and $\Delta(B)$ consist of simple normal crossing divisors.
Then the difference $\chi(X, \pi^\ast (\cF))) - \mu \cdot \chi(Y,\cF)$ is equal to:
\begin{align*}  \sum_{\underline{a}\; {\rm MF}} &(-1)^{\abs{\underline{a}}} \left( \underset{\abs{\underline{a}'} \ge 1}{\sum_{\underline{a}'\leq\underline{a}}}(- \constD_{\underline{a}-\underline{a}'}\constC_{\underline{a}'})(E_{R^{\underline{a}'}} - 1)\right)\chi(R^{\underline{a}}, \pi^\ast (\cF))) \\
&+  \sum_{\underline{a}\; {\rm NMF}} (-1)^{\abs{\underline{a}}}\left( \constC_{\underline{a}}(E_{R^{\underline{a}}} - 1) + \underset{\abs{\underline{a}'} \ge 1}{\sum_{\underline{a}'\leq\hat{\underline{a}}}}(- \constD_{\underline{a}-\underline{a}'}\constC_{\underline{a}'})(E_{R^{\underline{a}'}} - 1)\right)\chi(R^{\underline{a}},\Delta(R)', \pi^\ast (\cF))).
\end{align*}
The notation $D^{\underline{b}}$ is introduced in \ref{not:1}.(9), the constants $\constC_{\underline{b}}$ are introduced in \ref{not:2}.(1),  the terminology MF and NMF is from \ref{not:2}.(2), the constants $\constD_{\underline{b}}$ are introduced in \ref{not:2}.(4), the notation $E_{R^{\underline{a}}}$ is from \ref{not:ERA}, and the notation $\chi(R^{\underline{a}},\Delta(R)', \pi^\ast( \cF))$ is from \ref{not:logeuler}.
\end{cor}
\begin{proof}
The proof is the same as that for Corollary \ref{cor:leprim/imprim}.
The terms $-\constD_{\underline{a}-\underline{a}'}\constC_{\underline{a}'}(E_{R^{\underline{a}'}} - 1)$  account for the  contribution to the coefficient of $\chi(R^{\underline{a}}, \pi^\ast (\cF)))$ from the expansion of the terms $\chi(R^{\underline{a'}},\Delta' ,\pi^\ast (\cF)))$ where $\underline{a'}$ is MF.
The term $\constC_{\underline{a}}(E_{R^{\underline{a}}} - 1)$ in the NMF case accounts for the contribution of the term which already appears in Theorem \ref{thm:main}.
\end{proof}

\begin{prop}\label{prop:simpleMF}
If $\underline{a}$ is MF and $R^{\underline{a}} = R_1\cdots R_k$ then:
\[ (-1)^{\abs{\underline{a}}}  \underset{\abs{\underline{a}'} \ge 1}{\sum_{\underline{a}'\leq\underline{a}}}(- \constD_{\underline{a}-\underline{a}'}\constC_{\underline{a}'})(E_{R^{\underline{a}'}} - 1) = \constC_{\underline{a}}\prod_{i=1}^k (1-e_{R_i}). \]
\end{prop}
\begin{proof}
When $\underline{a}$ is MF we have by Proposition \ref{prop:constD} that:
\[ (-1)^{\abs{\underline{a}}}\constD_{\underline{a}-\underline{a}'}\constC_{\underline{a}'}= (-1)^{\underline{a}'}\constC_{\underline{a}}. \]
It follows that:
\[ (-1)^{\abs{\underline{a}}}  \underset{\abs{\underline{a}'} \ge 1}{\sum_{\underline{a}'\leq\underline{a}}}(- \constD_{\underline{a}-\underline{a}'}\constC_{\underline{a}'})(E_{R^{\underline{a}'}} - 1) = \constC_{\underline{a}}\underset{\abs{\underline{a}'} \ge 1}{\sum_{\underline{a}'\leq\underline{a}}} (-1)^{\abs{\underline{a}'}}(1-E_{R^{\underline{a}'}}). \]
A direct computation yields that:
\[ \underset{\abs{\underline{a}'} \ge 1}{\sum_{\underline{a}'\leq\underline{a}}}(-1)^{\underline{a}'}(1-E_{R^{\underline{a}'}}) = \prod_{i=1}^k (1-e_{R_i}) \]
from which the result follows.
\end{proof}
\begin{prop}\label{prop:simpleNMF}
If $\underline{a}$ is NMF and $\abs{\hat{\underline{a}}} \ge 1$ then:
\[ \constC_{\underline{a}}(E_{R^{\underline{a}}} - 1) + \underset{\abs{\underline{a}'} \ge 1}{\sum_{\underline{a}'\leq\hat{\underline{a}}}}(- \constD_{\underline{a}-\underline{a}'}\constC_{\underline{a}'})(E_{R^{\underline{a}'}} - 1) =
   \constC_{\underline{a}}(E_{R^{\underline{a}}} - E_{R^{\hat{\underline{a}}}}) .\]
\end{prop}
\begin{proof}
When $\underline{a}$ is NMF  and $\abs{\hat{\underline{a}}} \ge 1$ the same will be true for $\underline{a}-\underline{a'}$ for all choices of $\underline{a}'$ except $\underline{a}'=\hat{\underline{a}}$.
We thus have by Proposition \ref{prop:constDNMF} that:
\[ \constC_{\underline{a}}(E_{R^{\underline{a}}} - 1) + \underset{\abs{\underline{a}'} \ge 1}{\sum_{\underline{a}'\leq\hat{\underline{a}}}}(- \constD_{\underline{a}-\underline{a}'}\constC_{\underline{a}'})(E_{R^{\underline{a}'}} - 1) =
 \constC_{\underline{a}}(E_{R^{\underline{a}}} - 1)  - \constD_{\underline{a}-\hat{\underline{a}}}\constC_{\hat{\underline{a}}}(E_{R^{\hat{\underline{a}}}} - 1). \]
By noting that $ \constD_{\underline{a}-\hat{\underline{a}}}\constC_{\hat{\underline{a}}} = \constC_{\underline{a}}$ the result now follows immediately.
\end{proof}

\section{Handling Self Intersections}\label{sec:logeulerself}

The purpose of this section is to describe a method for interpreting the logarithmic Euler characteristic when there are self intersection terms. In particular we will show that  these can be viewed as a weighted sum of the Euler characteristics of the components of the self intersection.
The expressions one obtains are `non-canonical' but may be amenable to computation depending on the context. 

In order to carry out the procedure outlined here, one needs to have a good understanding of the Chow ring of the variety $X$. In particular the process may require a large number of relations consisting entirely of elements with simple normal crossings.
The reason we need an alternate approach is that though ideally we would be able to write:
\[ D^\ell \chern_i(\Omega_X(\log D)) = \chern_i(\Omega_{D^\ell}),  \]
this is simply not true if $\ell > 1$.
In order to handle this, we must have at least enough information to compute $D^\ell$.
In particular we will need to make use of relations:
\[ D\sim \sum_i u_{i}E_{i} \]
with the $E_i$ not being equal to any other divisor already in use, and with the total collection $E_{i}$, $D$, and every other divisor in use, having simple normal crossings.

\begin{lemma}
Let $X$ be a smooth projective variety and let $\cF$ be any coherent sheaf on $X$.
Suppose $\Delta$ is a collection of smooth divisors with simple normal crossings on $X$.
Fix $D \in \Delta$ and a relation 
\[ D\sim \sum_{i\in I} u_{i}E_{i} \]
with simple normal crossings as above.
We may rewrite:
\[  D^{\underline{a}} D^\ell Q_m(\chernchar(\cF); \chern_1(\Omega_X(\log\Delta)),\ldots,\chern_m(\Omega_X(\log\Delta)))  \]
as:
\[ D^{\underline{a}} D^{\ell-1} \sum_{k=1}^m  (-1)^{k-1}\constC_{(k-1)} \sum_i u_{i} E_i^k Q_{m-k+1}(\chernchar(\cF); \chern_1(\Omega_X(\log \Delta\cdot E_i)),\ldots,\chern_{m-k+1}(\Omega_X(\log \Delta\cdot E_i))).
\]
The constant $\constC_{(k-1)}$ is defined in \ref{not:2}.(1).
\end{lemma}
\begin{proof}
This follows immediately by a comparison between 
\[ Q_m(\chernchar(\cF); \chern_1(\Omega_X(\log \Delta)),\ldots,\chern_m(\Omega_X(\log \Delta)) )  \]
and
\[ \qquad Q_m(\chernchar(\cF); \chern_1(\Omega_X(\log \Delta\cdot E_i)),\ldots,\chern_m(\Omega_X(\log \Delta\cdot E_i))) \]
as in Theorem \ref{thm:eulvslogeul}.
\end{proof}

\begin{lemma}
Let $X$ be a smooth projective variety and let $\cF$ be any coherent sheaf on $X$.
Let $\Delta$ be a collection of smooth divisors with simple normal crossings on $X$.

Suppose we are given sufficiently many rules in the Chow ring of $X$ of the form:
\begin{align*}
 (a)\quad D_j \sim \sum_{i\in I_{ja}} u_iE_i \qquad \qquad\text{and} \qquad \qquad
 (b)\quad  E_j \sim \sum_{i\in I_{jb}} u_iE_i 
\end{align*}
expressed with respect to a collection of divisors  $E_i$ indexed by $I = \sqcup I_{ja}$, a universal family of shared indices and such that the total collection of divisors $D_i, E_j$ has simple normal crossings, then 
we may rewrite:
\[  D^{\underline{a}} D^\ell Q_m(\chernchar(\cF); \chern_1(\Omega_X(\log \Delta)),\ldots,\chern_m(\Omega_X(\log \Delta)))  \]
as a weighted sum of terms:
\[ D^{\underline{\tilde{a}}}E^{\underline{b}}Q_{n-\abs{\underline{\tilde{a}}}-\abs{\underline{b}}}(\chernchar(\cF); \chern_1(\Omega_X(\log \Delta\cdot E^{\underline{b}})),\ldots,\chern_{n-\abs{\underline{\tilde{a}}}-\abs{\underline{b}}}(\Omega_X(\log \Delta\cdot E^{\underline{b}}))) \]
with $b_i \leq 1$.
\end{lemma}
\begin{proof}
The key is to inductively apply the previous lemma.

We observe that at each application of the lemma we produce new terms of the form:
 \[ D^{\underline{a}'}E^{\underline{b}'}Q_{n-\abs{\underline{a}'}-\abs{\underline{b}'}}.
\]
However, each new term introduced either satisfies:
\begin{enumerate}
\item The number of self intersections has been decreased, or
\item The subscript on $Q_m$ has decreased.
\end{enumerate}
It follows that the inductive process terminates provided we have enough rules to carry it out.
\end{proof}

\begin{prop}
In the setting of the Lemma,
the coefficient of 
\[ D^{\underline{\tilde{a}}}E^{\underline{b}}Q_{m}(\chernchar(\cF); \chern_1(\Omega_X(\log \Delta\cdot E^{\underline{b}})),\ldots, \chern_m(\Omega_X(\log \Delta\cdot E^{\underline{b}})) \]
in the formal expansion of 
\[  D^{\underline{a}} Q_m(\chernchar(\cF); \chern_1(\Omega_X(\log \Delta)),\ldots,\chern_m(\Omega_X(\log \Delta)))  \]
is:
\[ \prod_i (u_i^{b_i} \constC_{(y_i)}) \]
where 
$ y_j = \abs{\underset{z}\cup I_{jz} \cap \underline{b}},$
that is, $y_j$ is the number of rules that must be used in the expansion of $E_j$.
\end{prop}
\begin{proof}

The appearance of the $\prod_i u_i^{b_i} \constC_{(y_i)}$ is apparent from the lemma, as these are precisely the terms that appear when we apply it. The only remaining question is the computation of $y_i$ based on the shape of $\underline{b}$. One readily checks the given formula.
\end{proof}

The only information we still lack about our expansion is which $E^{\underline{b}}$ actually appear. This depends on choices made during the inductive process, however, if one orders the rules one can obtain a systematic result. The following proposition is an immediate consequence of the inductive process.
\begin{prop}
Carrying out the inductive procedure as above, if the rules:
\begin{align*}
 (a)\quad D_j \sim \sum_{i\in I_{ja}} u_iE_i \qquad \qquad\text{and} \qquad \qquad
 (b)\quad  E_j \sim \sum_{i\in I_{jb}} u_iE_i 
\end{align*}
are ordered by $(a)$ and $(b)$ and we always select the first rule which does not conflict with choices already made then
the collection $E^{\underline{b}}$ which appear in the expansion are precisely those which satisfy:
\begin{enumerate}
\item $\abs{\underline{b}\cap I_{jc}} = 0,1$.
\item For each $D_j$, the number of $a$ for which $\abs{\underline{b}\cap I_{ja}} = 1$ is $a_j-1$.
\item $\abs{\underline{b}\cap I_{jc}} = 1$ and $c>0$ implies $\abs{\underline{b}\cap I_{j{c-1}}} = 1$.
\item $\abs{\underline{b}} \leq n-\abs{\tilde{\underline{a}}}$.
\end{enumerate}
\end{prop}

\begin{rmk}
Because $  D^{\underline{\tilde{a}}}E^{\underline{b}}Q_{m}(\chernchar(\cF); \chern_1(\Omega_X(\log \Delta\cdot E^{\underline{b}})),\ldots,\chern_m(\Omega_X(\log \Delta\cdot E^{\underline{b}}))) $
is computing a logarithmic Euler characteristic on $D^{\underline{a}}E^{\underline{b}}$ the above expansion gives a weighted sum of the logarithmic Euler characteristics for some representative cycles for various $D^{\underline{x}}$.
We note that $\prod_i u_i^{b_i}$ is somehow related to the coefficient that would have appeared had we been computing the self intersection whereas the coefficient $\prod_i  \constC_{(y_i)}$ is universal. None the less, we note that this process involves a number of non-canonical choices.

It is worth noting that by performing a further induction, as in Corollary \ref{cor:secondaryinduction}, we could replace the Logarithmic Euler characteristics with the actual Euler characteristics of the same components of the self intersections, simply with different weights.
\end{rmk}

\begin{ex}
Suppose we have relations:
\[  D \sim E_1+E_2  \qquad           E_1 \sim E_3  \sim E_4  \qquad E_2  \sim E_5 \sim E_6  \qquad E_3 \sim E_7\qquad E_5\sim E_8. \]
(Note that the implied relation $E_3\sim E_4$ (respectively $E_5\sim E_6$) is not being viewed as a rule for $E_3$ (respectively $E_5$).
Then we may carry out the procedure above as follows:
\begin{align*}
  D^2Q_2(&\chernchar(\cF); \chern_1(\Omega_X(\log \Delta)),\chern_2(\Omega_X(\log \Delta))) \\
              &= DE_1Q_2(\chernchar(\cF); \chern_1(\Omega_X(\log \Delta\cdot E_1)),\chern_2(\Omega_X(\log \Delta\cdot E_1)))\\&\qquad
                        +DE_1^2\constC_{(1)}Q_1(\chernchar(\cF);\chern_1(\Omega_X(\log \Delta\cdot E_1)) )
                        +DE_1^3\constC_{(2)}Q_0(\chernchar(\cF); )\\&\qquad
                        +DE_2Q_1(\chernchar(\cF); \chern_1(\Omega_X(\log \Delta\cdot E_2)),\chern_2(\Omega_X(\log \Delta\cdot E_2)))\\&\qquad
                        +DE_2^2\constC_{(1)}Q_1(\chernchar(\cF);\chern_1(\Omega_X(\log \Delta\cdot E_2)) )
                        +DE_2^3\constC_{(2)}Q_0(\chernchar(\cF); )\\
               &= DE_1Q_2(\chernchar(\cF); \chern_1(\Omega_X(\log \Delta\cdot E_1)),\chern_2(\Omega_X(\log \Delta\cdot E_1)))\\&\qquad
                    +DE_1E_3\constC_{(1)}Q_1(\chernchar(\cF);\chern_1(\Omega_X(\log \Delta\cdot E_1E_3)) )
                    +DE_1E_3^2\constC_{(1)}\constC_{(1)}Q_0(\chernchar(\cF); ) \\&\qquad
                    +DE_1E_3E_4\constC_{(2)}Q_0(\chernchar(\cF); ) \\&\qquad
                    +DE_2Q_1(\chernchar(\cF); \chern_1(\Omega_X(\log \Delta\cdot E_2)),\chern_2(\Omega_X(\log \Delta\cdot E_2)))\\&\qquad
                    +DE_2E_5\constC_{(1)}Q_1(\chernchar(\cF);\chern_1(\Omega_X(\log \Delta\cdot E_2E_5)) )
                    +DE_2E_5^2\constC_{(1)}\constC_{(1)}Q_0(\chernchar(\cF); ) \\&\qquad
                    +DE_2E_5E_6\constC_{(2)}Q_0(\chernchar(\cF); ).
\end{align*}  
Which we can ultimately express as:
\begin{align*}
               &= \chi(DE_1, \Delta', \cF|_{DE_1})  
                    +\constC_{(1)}\chi(DE_1E_3, \Delta', \cF|_{DE_1E_3})  \\&\qquad
                     +\constC_{(1)}\constC_{(1)}\chi(DE_1E_3E_7, \Delta', \cF|_{DE_1E_3E_7}) 
                     +\constC_{(2)}\chi(DE_1E_3E_4, \Delta', \cF|_{DE_1E_3E_4})  \\&\qquad
                     +  \chi(DE_2, \Delta', \cF|_{DE_2}) 
                     +\constC_{(1)}\chi(DE_2E_5, \Delta', \cF|_{DE_2E_5})  \\&\qquad
                   +\constC_{(1)}\constC_{(1)}\chi(DE_2E_5E_8, \Delta', \cF|_{DE_2E_5E_8})
                     +  \constC_{(2)}\chi(DE_2E_5E_6, \Delta', \cF|_{DE_2E_5E_6}).             
\end{align*}            
 In particular we can express the result purely as a sum of logarithmic Euler characteristics.       
\end{ex}

\section{Conclusions and Further Questions}

We have obtained a natural generalization of the Riemann-Hurwitz results to the algebraic Euler characteristic.
The formulas given are certainly more complicated than for the standard Euler characteristic.

It is natural to ask to what extent any of the results here can be generalized outside the context in which we are able to prove them.

\section*{Acknowledgments}

This work was primarily conducted while I was a Fields Postdoctoral researcher at Queen's University.
I would like to thank Prof. Mike Roth at Queen's University for teaching me many of the tools needed in this work as well as for his help improving this manuscript.
I would also like to thank the referee for many very useful comments.

\providecommand{\MR}[1]{}
\providecommand{\bysame}{\leavevmode\hbox to3em{\hrulefill}\thinspace}
\providecommand{\MR}{\relax\ifhmode\unskip\space\fi MR }
\providecommand{\MRhref}[2]{  \href{http://www.ams.org/mathscinet-getitem?mr=#1}{#2}
}
\providecommand{\href}[2]{#2}

\end{document}